\documentclass[10pt]{amsproc}
\usepackage{fullpage}
\usepackage{graphicx}
\usepackage{amssymb}
\usepackage{amsfonts}
\usepackage{amsthm}
\usepackage{amsmath}
\usepackage[shortalphabetic]{amsrefs}
\usepackage{xypic}
\usepackage{eucal}
\usepackage{pdfsync}
\usepackage{hyperref}
\usepackage[margin=1.25in]{geometry}

\def\Z{{\mathbb Z}}

\def\R{{\mathbb R}}

\def\K{{\mathbb K}}
\def\w{\mathcal{W}}

\def\cc{\mathcal{C}}
\def\dd{\mathcal{D}}
\def\pp{\mathcal{P}}
\def\bb{\mathcal{B}}

\def\b{\beta}
\def\r#1{\mathrm{#1}}

\def\mc#1{\mathcal{#1}}
\def\mc#1{\mathcal{#1}}

\def\w{\mathcal{W}}

\def\ainf{A_\infty}

\def\z2{\Z / 2\Z}

\def\id{\mathrm{id}}

\newtheorem{lem}{Lemma}
\newtheorem{alem}{Lemma}[section]
\newtheorem{prop}[lem]{Proposition}
\newtheorem{thm}[lem]{Theorem}
\newtheorem*{mainthm}{Main Theorem}
\newtheorem{cor}[lem]{Corollary}
\newtheorem{defn}[lem]{Definition}
\newtheorem{ques}[lem]{Question}

\theoremstyle{remark}
\newtheorem{rem}[lem]{Remark}

\numberwithin{equation}{section}
\begin{document}

\begin{abstract} 

    In all known explicit computations on Weinstein manifolds, the self-wrapped Floer homology of non-compact exact Lagrangian is always either infinite-dimensional or zero.  We show that a global variant of this observed phenomenon holds in broad generality: the wrapped Fukaya category of any positive-dimensional Weinstein (or non-degenerate Liouville) manifold is always either non-proper or zero, as is any quotient thereof. Moreover any non-compact connected exact Lagrangian is always either a ``(both left and right) non-proper object" or zero in such a wrapped Fukaya category, as is any idempotent summand thereof. We also examine criteria under which the argument persists or breaks if one drops exactness, which is consistent with known computations of non-exact wrapped Fukaya categories which are smooth, proper, and non-vanishing (e.g., work of Ritter-Smith). 

\end{abstract}

\def\oc{\mc{OC}}
\def\co{\mc{CO}}
\def\b{\mc{B}}
\def\z{ { \mathbb Z}}
\def\hhmf{\mathrm{R} \Gamma (\wedge^\bullet T_Y, [W - y, \cdot])}
\title{Categorical non-properness in wrapped Floer theory}
\author{Sheel Ganatra}
\maketitle
\section{Introduction}

The wrapped Fukaya category $\w(X)$ is an important invariant of a Liouville (exact
symplectic convex at infinity) manifold $X$, whose objects are (exact,
cylindrical at infinity) Lagrangian submanifolds and whose cohomological
morphisms are {\em wrapped Floer homology} groups.  Unlike the ordinary
Lagrangian Floer cohomology between a pair of Lagrangians, wrapped Floer homology allows
further contributions to the chain complex beyond just (the finitely many
transverse) intersection points of the Lagrangians, coming from Reeb chords
between the Legendrian boundaries at infinity.
The presence of chords in these complexes means that wrapped Floer homology could be infinite dimensional, but doesn't force infinite-dimensionality.\footnote{e.g., if there are only finitely many chords or the differential involves many cancellations} 
Indeed one of the motivations for their introduction 
\cite{abouzaidseidel} was precisely to mirror Ext groups of sheaves on
singular or non-compact varieties, which could be infinite
dimensional. This paper is concerned with the question of to what degree such wrapped Floer homology groups
{\em must} be infinite dimensional (if they are non-zero).

In recent years, there has been an
explosion of computations of wrapped Fukaya categories, especially in the
setting of {\em Weinstein manifolds} (those Liouville manifolds for which the
Liouville vector field $Z$, the symplectic dual of the primitive of the
symplectic form is gradient like for a Morse exhaustion function), where we
have strong general control of the category.  In all of the many known
computations to date on Weinstein manifolds (to the author's knowledge), the
following phenomenon occurs: the self-wrapped Floer cohomology of a non-compact
exact
(cylindrical at infinity) Lagrangian is either infinite dimensional or zero. 
An old folk question inquires
whether this could be always the case: 
\begin{ques}[Folk]\label{infiniteorzero}
    Let $L$ be a non-compact connected exact Lagrangian with Legendrian ends in a Liouville manifold $X$. Must the self-wrapped Floer homology $HW^*(L,L)$ be either infinite dimensional or zero?
\end{ques}
\begin{rem}
Question \ref{infiniteorzero} is very quickly false if one drops the connected
hypotheses (take a compact Lagrangian union a non-compact Lagrangian with zero wrapped Floer homology)
When considering the more general question of when $HW^*(K,L)$ could be finite non-zero, note that $HW^*(K,L)$ coincides with ordinary Floer cohomology $HF^*(K,L)$ (which is finite and could be non-zero) when one of $K$ or $L$ is compact, or more generally when $K$ and $L$ have Legendrian boundaries on different components at infinity.
However, the groups $HW^*(K,L)$ do appear to `frequently' be infinite for non-compact $K$, $L$.
\end{rem}

Question \ref{infiniteorzero} is the open-string analogue of (and generalizes by setting $L =
\Delta$) an even older and more well known folk question, about the nature of the often studied {\em symplectic
cohomology} $SH^*(X)$.  We recall that $SH^*(X)$, by definition the (Hamiltonian) Floer cohomology of a Hamiltonian with sufficiently large growth near infinity or the limit of Floer cohomologies of Hamiltonians which increasingly grow near infinity, is generated (for a nice choice of Hamiltonian) by the co-chains of $X$, a subcomplex, along with Reeb orbits on the boundary at infinity of $X$. 
\begin{ques}[Folk]\label{infiniteorzeroSH}
    For a non-compact Liouville manifold $X$, must $SH^*(X)$ be either infinite dimensional or zero?
\end{ques}
There are reasons to believe Questions \ref{infiniteorzero} and \ref{infiniteorzeroSH} 
are either rather difficult, or maybe have negative answers.
For one, if the answer to Question \ref{infiniteorzeroSH} is yes in general, it would follow
that Viterbo's ``acceleration map'' $H^*(X) \to SH^*(X)$ (induced by the inclusion of a subcomplex in the chain model mentioned above) cannot be an isomorphism (the so-called ``algebraic Weinstein conjecture''), further implying  
Weinstein's conjecture on the existence of a Reeb orbit on the boundary-at-infinity of $X$ for any non-degenerate contact form (which is currently unknown in this generality).  Similarly, an affirmative answer to Question \ref{infiniteorzero} for any $L$ would imply Arnold's chord conjecture on the existence of a Reeb self-chord on the Legendrian $\partial_{\infty} L$.
Even in the case of a cotangent bundle where one can reduce to algebraic topology via the isomorphism $SH^*(T^*Q) \cong H_{n-*}(\mc{L} Q) \neq 0$ \cite{viterboloop} where $\mc{L}Q$ denotes the free loop space\footnote{implicitly using twisted coefficients on one or both sides in case $Q$ is not Spin  \cite{kraghspin}*{Rmk. 1.3}} Question \ref{infiniteorzeroSH}, which then asks whether $H_{*}(\mc{L} Q)$ must always be infinite-dimensional for a closed $Q$, is to our knowledge unknown for all $Q$\footnote{though by direct computation it is known in many special cases, e.g., when $Q$ is simply connected or when $\pi_1 Q$ has infinitely many conjugacy classes; compare the discussion in \cite{ritterTQFT}*{\S 12.5}}. In all explicitly computed examples on Weinstein manifolds (again to our understanding) the alternatives posed in Questions \ref{infiniteorzero} and \ref{infiniteorzeroSH} hold. If one drops exactness however, the alternatives are known to fail, even if one retains non-compactness of the target (see e.g., \cite{ritter_negativeSH,rittersmith}).

Nevertheless, our main result is that a more global version of
the empirically noticed ``infinite or zero'' phenomenon in fact holds in broad generality for structural reasons. To state the main result, recall from \cite{ganatra1_arxiv} that a Liouville manifold is said to be {\em non-degenerate} if there exists a collection of Lagrangians in the wrapped Fukaya category satisfying Abouzaid's generation criterion \cite{abouzaid_generation}. We will treat this hypothesis, which has strong implications \cite{ganatra1_arxiv} (reviewed later), as a black box. A series of works implies every Weinstein manifold is non-degenerate \cite{CDGG, GPSstructural, ganatra1_arxiv, gao}. 

\begin{mainthm}
    \label{thm:nonproper}
    Let $X$ be a Weinstein manifold (or more generally a non-degenerate Liouville manifold in the sense of \cite{ganatra1_arxiv}) of dimension greater than zero.
    Then 
    \begin{enumerate}
        \item[(1)] The wrapped Fukaya category $\w(X)$ is either non-proper or it is zero. More generally, 
        \item[(1')] Any quotient of $\w(X)$ is non-proper or zero, as is more generally any category admitting a homological epimorphism from $\w(X)$.
        \item[(2)] Any open exact Lagrangian submanifold $L \in \w(X)$ (i.e., all components of $L$ are non-compact) is either a (both right and left) non-proper object or it is zero. More generally, 
        \item[(2')] Any idempotent summand (in the pre-triangulated split-closure $perf(\w(X))$) of an $L$ as in (2) is either (both right and left) non-proper or it is zero.
    \end{enumerate}
\end{mainthm}
Recall that an $\ainf$ category $\cc$ is {\em proper} over a field $\K$ if $\hom(X,Y) \in perf(\K)$ for all pairs of objects $X,Y \in \cc$ where, denoting by $Mod(\K)$ the category of all chain complexes over $\K$, $perf(\K) \subset Mod(\K)$ denotes the full subcategory of chain complexes with finite-dimensional total cohomology. An object $L \in \cc$ is {\em right (respectively left) proper} if $\hom_\cc(X,L)$ (respectively $\hom_{\cc}(L,X)$) lands in $perf(\K) \subset Mod(\K)$ for every $X \in \cc$. Properness of a category, or of a (left or right) module over a category such as the Yoneda modules $\hom_{\cc}(L,-)$ and $\hom_{\cc}(-,L)$, is a Morita invariant notion, inherited by (split-)generating subcategories. Statement (1) of the Main Theorem therefore implies that if $\w(X)\neq 0$ there exist $K, L \in \w(X)$ with $HW^*(K,L)$ infinite rank, including at least one such pair in any split-generating collection of objects. The more general (1') further precludes phenomena like the existence of a non-trivial proper orthogonal summands of the wrapped Fukaya category, which could a priori exist on a non-proper category. (2) correspondingly says that for any open exact $L$ which is not zero, there exist a $K, K'$ with $HW^*(K,L)$ and $HW^*(L,K')$ of infinite rank, including one such $K$ and $K'$ in any split-generating collection. (2') further prohibits such an $L$ from having a non-trivial idempotent summand which is right or left proper (in contrast, any compact exact Lagrangian $Q$ is both non-zero and right and left proper).
As explained in Corollary \ref{zeroorextalternative} and Remark \ref{zeroorexactcounterexample}, the Main Theorem also implies an alternative for the groups $SH^*(X)$ and $HW^*(L,L)$ (under the stated hypotheses) which is slightly weaker than, but partly explains the persistence in all known examples of, the alternatives raised by Questions \ref{infiniteorzeroSH} and \ref{infiniteorzero}.

\begin{rem}[Comparing Statements (1) and (2)]
    To compare Statements (1) and (2) of the Main Theorem in an instance,
    let's suppose $X$ has a (split-)generating collection $\{\Delta_i\}_{i=1}^k$,
    Then Statement (1) implies that if $\w(X)\neq 0$ at least one of the groups $HW^*(\Delta_r, \Delta_s)$ is infinite rank, and in particular 
    $HW^*(\coprod_i \Delta_i, \coprod_i \Delta_i):= \oplus_{r,s} HW^*(\Delta_r, \Delta_s)$ is infinite or zero, answering Question \ref{infiniteorzero} affirmatively for $L = \coprod_i \Delta_i$.

    Let's further assume that in the split-generating collection $\{\Delta_i\}_{i=1}^k$, all of the objects are open submanifolds (e.g., take the cocores of some Weinstein presentation which generate by \cite{CDGG, GPSstructural}) and furthermore discard all zero objects from the collection (so $\w(X) = 0$ iff $k=0$). Statement (2) then implies the stronger result that at least $k$ of the groups $\{HW^*(\Delta_r, \Delta_s)\}_{r,s=1}^k$ must be infinite rank, at least one for each $r$ and each $s$.\footnote{There are examples where exactly $k$ of the $HW^*(\Delta_r, \Delta_s)$ are infinite, e.g., when $HW^*(\Delta_i,\Delta_j) = 0$ for $i \neq j$.}

More generally, 
we see that Statement (1) actually follows from from Statement (2) provided one knows that if $\w(X) \neq 0$, at least one of the non-zero generators of $X$ must be a non-compact Lagrangian $L$, a fact which is true for Weinstein manifolds and more generally for all non-degenerate Liouville manifolds by known geometric arguments (for disc confinement reasons, Abouzaid's criterion cannot be satisfied by compact Lagrangians alone, compare \cite{Gcircleactions}). 
One can similarly deduce Statement (1') from a strengthened form of (2) whose statement is omitted from the theorem for simplicity (one shows the image of such an $L$ under categorical quotients is non-proper or vanishing) whose proof is also in this paper (see e.g., Lemma \ref{openobjectquotient}), along with similar checkable hypotheses about non-compact generators. Alternatively, e.g., Statement (1) can be extracted from Statement (2) applied to the diagonal $\Delta \subset X^- \times X$ (making this perspective precise requires some technical digressions which we opted to avoid here).
Although (2) and variants therefore imply (1) and variants, we nevertheless found it cleanest to state (and give superficially independent arguments for) (1) and (2),  to distinguish ``global'' (whole category) statements from ``local'' (single object) ones.
\end{rem}

\begin{rem}
    For Weinstein manifolds, one geometric source of quotients (resp. more general homological epimorphisms) comes from restricting to Weinstein subdomains \cite{GPSstructural}*{Prop. 8.15} (resp. Liouville subdomains which are independently Weinstein \cite{sylvanepi});  Statement (1') for such quotients/homological epimorphisms  is simply a consequence of Statement (1) for the subdomain.  
\end{rem}

\begin{rem}
    From a mirror symmetry lens Statement (1') implies that if there is an HMS equivalence $\w(X) = coh(Y)$, then no component of $Y$ can be simultaneously smooth and proper.
\end{rem}

The proofs of the statements in the Main Theorem appeal to the structure and properties of some TQFT operations in wrapped Floer theory, particularly the ``degeneracy'' of the geometric closed respectively open string copairing for Liouville manifolds respectively their exact open Lagrangian submanifolds \cite{ritterTQFT}, and its interplay with the fact that (under the hypotheses imposed) $\w(X)$ is a {\em (weak) smooth Calabi-Yau category} \cite{ganatra1_arxiv}.\footnote{$\w(X)$ further has a (non-weak) smooth Calabi-Yau structure \cite{Gcircleactions}, but we only use the weak structure here.} Broadly the proof of (1), respectively (2) goes as follows. Suppose $\w(X)$ (resp. $L \in \w(X)$) is proper (resp. left or right proper). Then, we observe that (given $X$ is non-degenerate) there must exist a pairing on $SH^*(X)$ (respectively $HW^*(L,L)$) that is compatible with the copairings in the sense of satisfying the snake relation in TQFT. 
This implies the copairing is non-degenerate which (by simultaneous degeneracy) implies $\w(X)$ (respectively $L$) is zero. The stronger (1') and (2') follow from translating the degeneracy of geometric copairings into the degeneracy of certain algebraic copairings which exist on $\w(X)$ respectively $L$ under the stated hypotheses, a property (here called {\em algebraic openness}) that both implies the non-properness or vanishing alternative and is inherited by quotients (respectively summands).

One can ask to what degree the Main Theorem holds in other studied non-compact settings\footnote{in compact cases, Fukaya categories are proper by construction}. If one turns on a stop, {\em partially wrapped Fukaya categories} (or wrapped Fukaya categories of Liouville {\em sectors} like $T^* \R^n$) can be proper non-zero (it may be interesting to articulate a relevant criterion which sometimes fails and sometimes holds). 
Another question is to what extent the result holds in non-exact (non-compact) cases when all of the structures are defined. Here we note that either of the key properties used, non-degeneracy of $\w(X)$ and the degeneracy of the copairings, could independently fail or hold in non-exact non-compact settings. For instance, as spelled out in Remark \ref{rem:degeneracylemmanonexact}, work of Albers-Kang \cite{alberskang} implies the failure of degeneracy of the copairing for negative line bundles, making our arguments inapplicable in such settings. Computations by Ritter/Ritter-Smith \cite{ritter_negativeSH, rittersmith} show that the outcome of Main Theorem indeed does also fail in such settings, despite a version of nondegeneracy holding.  
See Remarks \ref{rem:degeneracylemmanonexact}, \ref{rem:degeneracylemmanonexactlagrangian}, and \ref{rem:nonexact} for more exploration of non-exact settings, including structural hypotheses under which the result would continue to hold.

The paper is organized as follows. In \S \ref{sec:degeneracylem}, we review the definition of the geometric closed and open-string copairings and their degeneracy property, following \cite{ritterTQFT}*{Thm 13.3} with a slightly different exposition. In \S \ref{sec:copairingcy}, we turn towards the more abstract setting of weak smooth Calabi-Yau categories and review the fact that (as a special case of the idea that these categories carry suitable open-closed TFT operations) such categories carry algebraic ``closed-string'' and ``open-string'' copairings (on Hochschild homology and the endomorphisms of any object respectively). We also show that --- when these copairings are degenerate in a suitable sense --- the category/object in question, and in fact any quotient/summand thereof, must be non-proper or zero, due to the automatic existence of a dual compatible pairing when proper (as constructed in \cite{shklyarov_pairing, bravdyckerhoff}, though the object case requires some further arguments spelled out here). In \S \ref{sec:geometrytoalgebra} we tie the threads together and recall that the wrapped Fukaya category indeed has a canonical (weak) smooth Calabi-Yau structure \cite{ganatra1_arxiv}, whose associated algebraic closed and open-string copairings recover the geometric closed and open-string copairings (and are therefore degenerate) by \cite{rezchikov} and \cite{ganatra1_arxiv} respectively (again with some further spelling out in the open-string case). The proof of the Main Theorem appears in \S \ref{sec:proof}. An Appendix collects a few additional needed results about weak smooth Calabi-Yau structures which we expect are well known (but for which we could not find a precise reference).

\subsection*{Conventions}
For simplicity in our main argument below, we work over a fixed arbitrary field $\K$;
however, the Main Theorem also holds over $\Z$ (where properness is defined in terms of taking values in perfect dg $\Z$-modules in $perf(\Z) \subset Mod(\Z)$, the subcategory of those chain complexes which are quasi-isomorphic to a bounded complex of finitely-generated free abelian groups), as can be most simply deduced from the case of fields by the universal coefficient theorem.
We are agnostic about grading structures, e.g.,
Floer-theoretic and categorical structures in sight could be $\Z/2$ graded or $\Z$ graded
or have some other grading (degrees and degree shifts should be interpreted mod 2 in the 
former case).  While at one point in \S \ref{sec:degeneracylem} we appeal to cohomological degrees, this degree need not
coincide with gradings in wrapped Floer theory. We indicate the $\K$-linear
dual of a vector space $V$ by $V^{\vee}$.
All (dg or $\ainf$) categories considered in this paper are (cohomologically graded and) cohomologically
unital, meaning there are cohomological identity morphisms.

\subsection*{Acknowledgments}

I am grateful to Mohammed Abouzaid, Alexander Efimov, Yuan Gao, Oleg Lazarev, Yank\i\ Lekili,
Daniel Pomerleano, Vivek Shende, Zachary Sylvan, Alex Takeda, Sara Venkatesh,
and Wai-Kit Yeung for helpful conversations and comments. I was partly
supported by NSF grant DMS--1907635.

\section{Degeneracy of the geometric copairing}\label{sec:degeneracylem}

In this section we review the construction of geometric copairings on symplectic cohomology and wrapped Floer cohomology, along with their relevant degeneracy property in the open exact case, a key ingredient in the main result. The formulation of the degeneracy property also makes use of unital algebra structures on these groups.

First, let us recall that the symplectic cohomology $SH^*(X)$ of any Liouville manifold
admits various TQFT operations coming from counting maps from surfaces with at least one output \cite{ritterTQFT}. The pair of pants determines a {\em algebra} with {\em unit} given by the count of a cap (our convention following Ritter is that these operations are in degree zero).  
It also has a canonical {\em (geometric) closed-string copairing}
\[
    c_{SH}: \K \to SH^*(X) \otimes SH^*(X)[2n].
\]
given by counting cylinders where both ends are thought of as an output. 
There does not typically exist a pairing (or more generally any TQFT operations
coming from surfaces with no outputs), a feature of the non-compactness of the
situation. There exist technical problems (a `failure of the maximum principle') in setting up the moduli spaces for a pairing, reflecting the more fundamental point that if a pairing did exist, the usual snake relation in TFT would imply that both the pairing and copairing are non-degenerate and that $SH^*(X)$ is finite dimensional (which is not always true by computation). Let us recall that any copairing $c: \K \to V \otimes V[k]$ induces a map $c^*: V^{\vee}[-k] \to V$ sending $\phi \mapsto (\phi\otimes id)(c(1))$, we say $c$ is {\em non-degenerate} if this map $c^*$ is an isomorphism. 
For our purposes, the crucial feature of the closed-string copairing in the exact non-compact setting is the following {\em degeneracy property}, which in fact provides an algebraic obstruction to the existence of a compatible pairing unless $SH^*(X) = 0$:
\begin{thm}[``Degeneracy Lemma'', Ritter \cite{ritterTQFT} Thm 13.3]\label{degeneratecopairing}
    If $X^{2n}$ is Liouville, then the image of $c^*_{SH}$ of any class in $SH^*(X)^{\vee}[-2n]$ is nilpotent in $SH^*(X)$. In particular, $c_{SH}$ is non-degenerate if and only if $SH^*(X) = 0$.
\end{thm}
We recall the main thrust of the argument, with a somewhat different exposition:
\begin{proof}
    It is well known that the copairing in $SH^*(X)$ (which by \cite{ritterTQFT} can be computed in terms of the TQFT {\em coproduct} applied to the unit) factors through constant loops (compare \cite{ritterTQFT}*{Thm 6.10}), in the sense that it coincides with composition $\K \to H^*(X) \otimes H^*(X)[2n] \to SH^*(X) \otimes SH^*(X)[2n]$ where $H^*(X) \to SH^*(X)$ is Viterbo's acceleration map and the first map is the ``topological copairing'', i.e., the image of 1 under the wrong way map $\Delta_!(1)$ (here $\Delta: X \to X \times X$ is the diagonal embedding and we are implicitly identifying $H^*(X \times X) \cong H^*(X) \otimes H^*(X)$ using K\"{u}nneth). As a result, the induced map $c^*_{SH}$ factors as $SH^*(X)^{\vee}[-2n] \to H^*(X)^{\vee}[-2n] \to H^*(X) \to SH^*(X)$, where using Poincar\'e duality to identify $H^*(X)^{\vee}[-2n] \cong H^*_c(X)$, the intermediate map can be identified with the canonical map from compactly supported to ordinary cohomology. Since every component of $X$ is non-compact, the map $H^*_c(X)$ to $H^*(X)$ has image which is nilpotent with respect to the cup product (as $H^*_c(X)$ vanishes in cohomological degree zero).  
Now the map $H^*(X) \to SH^*(X)$ is an algebra map when $X$ is Liouville (using the cup product on $H^*(X)$ and the pair of pants product on $SH^*(X)$), hence preserves the condition of being nilpotent.
For the second statement: if $c^*_{SH}$ is an isomorphism, one learns that $ 1 \in SH^*(X)$ is nilpotent, hence $SH^*(X)=0$.

\end{proof}
\begin{rem}[On degeneracy of copairings and Rabinowitz Floer homology]\label{rem:rabinowitz}
    Theorem \ref{degeneratecopairing} is stated somewhat differently from the original reference \cite{ritterTQFT}*{Thm 13.3}, which makes the equivalent statement that if $RFH^*(X) = 0$ if and only if
    $SH^*(X)=0$, where $RFH^*(X)$ denotes the {\em Rabinowitz Floer homology} of $X$ as defined in \cite{CF}. We recall that
    by Poincar\'{e} duality in Floer theory, the map $c_{SH}^*$ can be identified with the {\em continuation map} $c: SH_*(X) \to SH^*(X)$ where $SH_*(X)$ is the {\em symplectic homology} (the Floer theory of a Hamiltonian of sufficiently {\em negative} growth, or the inverse limits of Floer cohomology over Hamiltonians with negative linear growth near infinity). In turn, the map $c$ appears between $SH_*(X)$ and $SH^*(X)$ in a long-exact sequence with $RFH^*(X)$ \cite{CFO}, allowing for the translation to Theorem \ref{degeneratecopairing} as stated. \\
\end{rem}

\begin{rem} \label{rem:degeneracylemmanonexact}
It's worth examining how the degeneracy Lemma (Theorem
\ref{degeneratecopairing}) could fail or alternatively succeed on a non-exact $X$ (assume one has a definition of $SH^*(X)$ with its TQFT structures), seeing as it is a key component in categorical ``non-properness or vanishing'' phenomena. If  $X$ is compact the argument breaks because the difference between $H_c^*(X)$ and $H^*(X)$ vanishes; indeed the copairing is non-degenerate in such cases as one can construct a pairing. For more general non-exact $X$ the presence of $J$-holomorphic
    spheres means that the acceleration map from $H^*(X)$ to $SH^*(X)$, a non-compact version of the PSS morphism \cite{pss}, should be an algebra map for the
{\em quantum product} (rather than classical cup product) on $H^*(X)$.
Classes of positive cohomological degree may no longer be nilpotent
for the quantum product on $H^*(X)$, which could break the argument (even if $X$ is non-compact).

On the other hand, in the absence of strictly positive Chern number holomorphic spheres, the quantum product
should only preserve or increase cohomological degree, meaning (provided each component of $X$ is non-compact)
the image of $H^*_c(X) \to H^*(X)$ should remain nilpotent and Theorem
\ref{degeneratecopairing} would seem to be applicable.  For instance, this
would seem to be the case on open non-exact symplectic manifolds $X$ for which $c_1(X) = 0$ and $SH^*(X)$ is defined (e.g., 
on a toric Calabi-Yau). We thank Abouzaid for discussions about this case.

Note that on some negative line bundles, which (are non-compact and) have positive Chern number holomorphic spheres, Theorem \ref{degeneratecopairing} indeed is known to fail: Ritter showed $SH^*(X)$ can be non-zero \cite{ritter_negativeSH} while  Albers-Kang \cite{alberskang}*{\S 4} compute that nevertheless in such cases Rabinowitz Floer homology vanishes, or equivalently $c_{SH}$ is non-degenerate (see Remark \ref{rem:rabinowitz}).
\end{rem}

Turning to the open-string setting, we recall that there is correspondingly an open TQFT structure on wrapped Floer cohomology (also developed in \cite{ritterTQFT}*{Thm. 6.13}). For any triple of exact Lagrangians $L_0, L_1, L_2 \in \w(X)$, counting maps from a disc with two inputs and one output induces a composition map in wrapped Floer cohomology $[\mu^2]: HW^*(L_1, L_2) \otimes HW^*(L_0,L_1) \to HW^*(L_0, L_2)$ a cohomological shadow of the chain-level $\ainf$ structure (see e.g., \cite{abouzaid_generation}).  Counting maps from an (unstable) disc with one output induces cohomological identity morphisms $[id_L] \in HW^*(L,L)$ for each $L$ which along with $[\mu^2]$ (restricting to $L_0 = L_1 = L_2 = L$) again equip $HW^*(L,L)$ with the structure of a unital $\K$-algebra.  There is also a {\em geometric open-string copairing} on the wrapped Floer cohomology which cohomologically gives an element
\begin{equation}\label{opencopairing}
    c_{K,L}: \K \to HW^*(K,L) \otimes HW^*(L,K)[n]
\end{equation}
defined for any pair of exact Lagrangians $K$ and $L$. We'll equate $c_{K,L}$ with $c_{K,L}(1)$ and further indicate $c_{L,L}$ by simply $c_L$. The element \eqref{opencopairing} is constructed in exactly the same manner as the closed-string copairing, by counting maps from an (unstable) disc with two outputs (or equivalently from a stable disc with two outputs and one interior point with fixed cross ratio, with the interior point unconstrained); this is a special case of the construction in e.g., \cite{ritterTQFT}*{Thm 6.13}, or a variant of the coproduct map in \cite{abouzaid_generation} with no inputs).

\begin{thm}[``Open-string degeneracy Lemma'', open-string analogue of Ritter \cite{ritterTQFT} Thm 13.3]\label{degeneratecopairingopenstring}
    Let $L \in \w(X)$ be an exact open Lagrangian (i.e., every component of $L$ is non-compact). 
    Then the image of $c^*_{L}$ of any class in $HW^*(L,L)^{\vee}[-n]$ is nilpotent in $HW^*(L,L)$. In particular, $c_{L}$ is non-degenerate if and only if $L=0$ (i.e., equivalently $HW^*(L,L) = 0$).
\end{thm}
\begin{proof}
    The argument is completely identical to the proof of Theorem \ref{degeneratecopairing}. The salient points are: the open-string copairing factors through constant paths, forcing the induced map $c_L^*$ to factor as $HW^*(L,L)^{\vee}[-n] \to H^*(L)^{\vee}[-n] \to H^*(L) \to HW^*(L,L)$. Since every component of $L$ is non-compact, the image of the middle map consists of nilpotent elements with respect to the cup product. Since $L$ is exact the rightmost map is an algebra map (for the cup product and $[\mu^2]$).
\end{proof}

Again this degeneracy property can be thought of as providing (unless the group in question vanishes) a obstruction to the existence of a compatible pairing.

\begin{rem}\label{rem:degeneracylemmanonexactlagrangian}
    Continuing Remark \ref{rem:degeneracylemmanonexact}, we examine how Theorem \ref{degeneratecopairingopenstring} could fail or succeed if the hypotheses are relaxed (supposing ordinary and wrapped Floer homology are still well-defined). If $L$ is a compact (say exact) Lagrangian $L$, the map  $H_c^*(L) \to H^*(L)$ is an isomorphism, breaking the above proof. Indeed, standard Floer theoretic methods allow one to construct a pairing $HF^*(L,L) \otimes HF^*(L,L) \to \K[-n]$; which can be used (along with the usual snake relation) to establish non-degeneracy of the geometric open-string copairing even though one has $HF^*(L,L) \cong H^*(L) \neq 0$. In the technical construction of such moduli spaces, one sees that there is extra freedom in the choice of ``Floer data'' which allows for operations without outputs when one or both Lagrangian is compact (without violating the maximum principle). One expects for the same reasons the failure of Theorem \ref{degeneratecopairingopenstring} to persist for compact $L$ in more general non-exact settings, provided such an $L$ has well-defined non-vanishing Floer theory.

    Supposing $L \subset X$ is non-compact and non-exact (whether $X$ is exact or not), and all of the above structures are defined. A similar argument would imply the map $c_{L}^*$ now factors through the map from $HF^*(L^-, L)$ to $HF^*(L^+, L)$ where $HF^*(L^-, L)$ is possibly a deformation (by counting discs) of $H^*_c(L)$ and $HF^*(L^+, L)$ similarly possibly a deformation of $H^*(L)$. These groups (if defined) may not coincide with ordinary (compactly supported or usual) cohomology, and even if there is a coincidence of vector spaces, the algebra structure may still be different/deformed. Theorem \ref{degeneratecopairingopenstring} may continue to hold if the groups are still $\Z$-graded or there is more generally some ``non-positivity'' control over Maslov indices of discs (paralleling Remark \ref{rem:degeneracylemmanonexact}).
\end{rem}

\section{Copairings in smooth Calabi-Yau categories}\label{sec:copairingcy}

The main goal of this section is to recall two types of natural copairings which exist on a (weak) smooth Calabi-Yau category, one on Hochschild homology (the ``closed string algebraic copairing'') and one on the endomorphisms of any object (the ``open string algebraic copairing'') whose degeneracy (in a suitable sense, which we term being ``algebraically open'') implies non-properness or vanishing phenomena. The closed string copairing was first introduced by Shklyarov \cite{shklyarov_pairing} and, though we do not know of a reference specifically studying the open-string copairing, it can be extracted as a special ``length 0'' case of known maps as we explained below.  We review basic definitions of such categories and these copairings in \S \ref{subsec:copairings}.
In \S \ref{subsec:nondegcopairing}, we review the general result that states that when such a category (respectively an object in such a category) is proper (respectively left/right proper) then one can automatically construct a pairing compatible with the (open or closed string) copairing, implying by a suitable snake relation non-degeneracy of the copairing. In the closed-string setting this latter result is again due to Shklyarov {\em loc. cit.}, and in the open-string setting we show this desired statement can be extracted with some additional work from a more general result of Brav-Dyckerhoff \cite{bravdyckerhoff} that proper objects in a smooth Calabi-Yau category inherit a ``proper Calabi-Yau structure''. In \S \ref{subsec:openness} we introduce the algebraic openness condition on a Calabi-Yau category or an object therein, show using the previous section that it implies non-properness or vanishing phenomena, and also show that it transfers along quotients and idempotent summands respectively.

\subsection{The algebraic copairings}\label{subsec:copairings}
Recall that for any
dg or $\ainf$ category $\cc$ over $\K$ (with cohomological composition denoted $[\mu^2]$), one has a dg category $[\cc,\cc]$ of ($\ainf$)
$\cc\!-\!\cc$ bimodules (defined as bilinear $\K$-linear $\ainf$ functors from $\cc^{op} \times
\cc$ to chain complexes over $\K$); see \cite{seidelnatural, ganatra1_arxiv, shklyarov_pairing, bravdyckerhoff, sheridanformulae} 
for some references. We indicate morphisms (derived hom) by
$\hom_{\cc\!-\!\cc}(-,-)$,  the operation of (derived) one-sided tensor product
by $-\otimes_{\cc}-$ (this produces another bimodule) and the operation of (derived)
two-sided tensor product by $-\otimes_{\cc\!-\!\cc}-$ (this produces a
chain-complex), and for our purposes isomorphisms respectively inverses mean homological isomorphisms/inverses. 
In the category $[\cc,\cc]$, there are some canonical objects, the {\em
diagonal bimodule} $\cc_{\Delta}:= \hom_{\cc}(-,-)$ (sometimes just indicated by $\cc$) and the {\em representable}
(or Yoneda) bimodules $Y_{K,L}:= hom_{\cc}(K,-) \otimes hom_{\cc}(-,L)$.
Evaluating at a pair of objects $ev_{K,L}: \mc{B} \mapsto \mc{B}(K,L)$ induces
a dg functor from bimodules to chain complexes, and in particular there is a
chain map $(ev_{K,L})_*: \hom_{\cc\!-\!\cc}(\mc{B}, \mc{B}') \to
\hom_{Mod(\K)}(\mc{B}(K,L), \mc{B}'(K,L))$, sending a morphism $f$ to $(ev_{K,L})_*f  = f_{K,L}$ (which in the bar model can be
thought of as projection to a length zero quotient complex. 

The {\em Hochschild homology} $\r{HH}_*(\cc)$ of $\cc$ is the (cohomology of the) self-bimodule (two sided) tensor product of the diagonal bimodule, whereas the {\em Hochschild cohomology} $\r{HH}^*(\cc)$ is the (cohomology) of the endomorphisms of the diagonal bimodule
(our grading conventions follow \cite{abouzaid_generation, ganatra1_arxiv}).
More generally, one can take the Hochschild homology (= bimodule tensor product with diagonal of) and Hochschild cohomology (=endomorphisms from diagonal to -) of any bimodule $\mc{B}$; the results are denoted $\r{HH}_*(\cc, \bb)$ and $\r{HH}^*(\cc, \bb)$ and generalize the previous case by setting $\bb=\cc_{\Delta}$. There is a natural {\em cap product action} $\r{HH}^*(\cc,\bb) \otimes \r{HH}_*(\cc) \stackrel{\cap}{\to} \r{HH}_*(\cc, \bb)$, coming from the functoriality of bimodule tensor product in the bimodules being considered. For any object $X \in \cc$, there are natural maps $H^*(\bb(X,X)) \to \r{HH}_*(\cc, \bb)$ and $\r{HH}^*(\cc, \bb) \to H^*(\bb(X,X))$ which on the level of bar complexes can be modeled as inclusion of or projection to the length 0 terms; hence we'll call the maps {\em inclusion} respectively {\em projection}. In the case of diagonal coefficients, 
the map $\r{HH}^*(\cc) \to H^*(\hom_{\cc}(X,X)) = H^*(\cc_{\Delta}(X,X))$ is a unital algebra map, and the map
map $H^*(\hom_{\cc}(X,X)) \to \r{HH}_*(\cc)$ can be used to
define the {\em Chern character $ch_X \in \r{HH}_0(\cc)$} of $X$ to be the image of the identity $[id_X]$. (Implicitly all $\ainf$ categories we consider are cohomologically unital, i.e., they have cohomological identity elements).

We say $\cc$ is {\em (homologically) smooth} if its diagonal bimodule $\cc_{\Delta}$ is a {\em perfect bimodule}, meaning it is split-generated by  representable (or Yoneda) bimodules. Denoting $perf(\cc\!-\!\cc)$ the subcategory of perfect bimodules, Morita invariance implies that the bilinear embedding $\cc^{op} \times \cc \to perf(\cc\!-\!\cc)$ induces an isomorphism $\r{HH}_*(\cc^{op}) \otimes \r{HH}_*(\cc) \cong \r{HH}_*(perf(\cc\!-\!\cc))$; furthermore there is an isomorphism $\r{HH}_*(\cc^{op}) \cong \r{HH}_*(\cc)$. When $\cc$ is smooth, the Chern character of the diagonal bimodule $ch_{\cc_{\Delta}}$ therefore gives an element in $\r{HH}_*(perf(\cc\!-\!\cc)) \cong \r{HH}_*(\cc) \otimes \r{HH}_*(\cc)$, or equivalently a map from $\K$ to $\r{HH}_*(\cc) \otimes \r{HH}_*(\cc)$ which we call the {\em algebraic closed-string copairing} $c_{alg}$ \cite{shklyarov_pairing}.

\begin{rem}
    Note that while the algebraic closed-string copairing only requires smoothness to define, we will require a (weak smooth) Calabi-Yau structure to formulate the relevant ``algebraic openness'' degeneracy condition; see Definition \ref{def:algebraicallyopen}. 
\end{rem}

Recall that
for any bimodule $\pp$, there is a {\em bimodule dual} $\pp^!$, whose value on a pair of objects $(K,L)$ is $\pp^!(K,L):= \hom_{\cc\!-\!\cc}(\pp, Y_{K,L})$. We call $\cc_{\Delta}^!$ the {\em inverse dualizing bimodule} (note $H^*(\cc_{\Delta}^!(K,L)) = \r{HH}^*(\cc, Y_{K,L})$), and note there is a canonical map $ev: \r{HH}_*(\cc) \to H^*(\hom_{\cc\!-\!\cc}(\cc^!, \cc))$ sending an element $\alpha$ to a morphism of bimodules which (on the level of maps of chain complexes) corresponds to capping with $\alpha$. $ev$ is an equivalence for smooth $\cc$ (analogous to the fact that the canonical map for vector spaces $V \otimes V \to \hom(V^*, V)$ is an equivalence when $V$ is finite). 
A {\em (weak) smooth Calabi-Yau structure} of dimension $n$ on a smooth category $\cc$ is an element $\sigma \in \r{HH}_{-n}(\cc)$ such that $ev(\sigma)$ is an isomorphism of bimodules $\cc^![n] \to \cc$; on the cohomology level for a pair $(K,L)$ the map $ev(\sigma)$ is the operation $-\cap \sigma: \r{HH}^*(\cc, Y_{K,L}) \stackrel{\cong}{\to} \r{HH}_{*-n}(\cc, Y_{K,L}) \cong \cc_{\Delta}(K,L)$. It follows that the cap product map $-\cap \sigma: \r{HH}^*(\cc, \bb) \stackrel{\cong}{\to} \r{HH}_{*-n}(\cc, \bb)$ is an isomorphism for any bimodule, as it can be realized by taking the bimodule tensor product of $ev(\sigma)$ with $\bb$, along with the isomorphism $H^*(\cc_{\Delta}^! \otimes_{\cc\!-\!\cc} \bb) \cong \r{HH}^*(\cc, \bb)$ that exists when $\cc$ is smooth. A {\em weak smooth Calabi-Yau category} is a pair $(\cc, \sigma)$ of a (dg/$\ainf$) category equipped with a weak smooth Calabi-Yau structure.
For a weak smooth Calabi-Yau category $(\cc, \sigma)$,  we will denote by $CY_{\cc}:  \r{HH}_{*-n}(\cc) \stackrel{\cong}{\to} \r{HH}^*(\cc)$ the inverse equivalence to capping with $\sigma$.

We now turn to defining the open-string copairing associated to a weak Calabi-Yau structure. First we define a copairing $c^f_{K,L} \in  H^*(\hom_{\cc}(K,L)) \otimes_{\K} H^*(\hom_{\cc}(L,K))[n]$ associated to any closed morphism of bimodules $f \in \hom^n_{\cc\!-\!\cc}(\cc_{\Delta}, \cc^![n])$, called the {\em copairing shadow} of $f$. 
The morphism $f$ induces, for every pair of objects $(L,L)$ a chain map $f_{L,L}: \hom_{\cc}(L,L) \to \cc^![n](L,L)$; taking the image of any cocycle $id_L$ cohomologically representing the identity (any `cohomological unit') gives a cocycle $f_{L,L}(id_L) \in  \cc^!(L,L) = \hom_{\cc\!-\!\cc}(\cc, \hom_{\cc}(-, L) \otimes_{\K} \hom_{\cc}(L,-))$; now such a bimodule morphism induces, for any object $K$, a map of chain complexes $(f_{L,L}(id_L))_{K,K}: \hom_{\cc}(K,K) \to \hom_{\cc}(K,L) \otimes_{\K} \hom_{\cc}(L,K)$; taking the image of $id_K$ induces a well defined cohomology class $[(f_{L,L}(id_L))_{K,K}(id_K)] \in H^*(\hom_{\cc}(K,L)) \otimes_{\K} H^*(\hom_{\cc}(L,K))$ which we define to be the copairing shadow $c^{f}_{K,L}$ (this only depends on the cohomology class $[f]$). 
Let's observe first that by Morita invariance the categories of bimodules associated to $\cc$ and its pre-triangulated split-closure $perf(\cc)$ are naturally quasi-equivalent in a way identifying (up to quasi-isomorphism) the diagonal bimodules and inverse dualizing bimodules. Hence, such an $f$ induces a canonical copairing shadow $c_{K,L}^f$ for any $K,L \in perf(\cc)$, coinciding with the previously defined copairing shadow for objects in $\cc$.

Finally, given a weak smooth Calabi-Yau structure $\sigma$ on $\cc$ so $ev(\sigma)$ gives an isomorphism of bimodules $\cc^![n] \stackrel{\sim}{\to} \cc$, we define the {\em algebraic open-string copairing} associated to $(\cc, \sigma)$ to be the copairing shadow of the inverse to $ev(\sigma)$: $c_{K,L}^\sigma := c_{K,L}^{ev(\sigma)^{-1}}$. As before we will note that these are defined for any $K,L \in perf(\cc)$ and we will abbreviate $c_L^{\sigma}:= c_{L,L}^\sigma$.

\subsection{Properness implies non-degeneracy}\label{subsec:nondegcopairing}
We will recall here that one has strong control over the closed respectively open string algebraic copairings on a weak smooth Calabi-Yau category when the given category respectively object is proper. The closed string case is directly in the literature and doesn't require the Calabi-Yau structure:
\begin{prop}[Shklyarov \cite{shklyarov_pairing}, where the result is attributed to Kontsevich-Soibelman]\label{shklyarovprop}
    When a smooth category $\cc$ is further proper, then its algebraic closed-string copairing is perfect, meaning $c^*_{alg}$ is an isomorphism.
\end{prop}
\begin{proof}[Sketch]
    When $\cc$ is proper, {\em loc. cit.} proves $\r{HH}_*(\cc)$ carries a canonical pairing (which can be thought of as essentially the pushforward map induced by $\hom: \cc^{op} \otimes \cc \to perf(\K)$ after using Morita invariance to deduce $\r{HH}_*(perf(\K)) = \r{HH}_*(\K) = \K$ and appealing to the K\"{u}nneth formula for $\r{HH}_*$ and op-invariance $\r{HH}_*(\cc^{op}) \cong \r{HH}_*(\cc)$). {\em loc. cit.} further proves that when $\cc$ is smooth, the closed-string algebraic copairing and pairing are related by the usual {\em snake relation}, implying in particular that the map $V \to V^* \to V$ induced first by the pairing and then the copairing is the identity map. This implies $c^*_{alg}: V^* \to V$ is an isomorphism as desired.
\end{proof}

Turning to the case of open-string copairings, we shall now explain that similarly, (right or left) properness of an object $L$ implies non-degeneracy of its algebraic copairing. The main idea is to appeal to a more general result of Brav-Dyckerhoff \cite{bravdyckerhoff}*{Thm 3.1} that smooth Calabi-Yau structures (``left Calabi-Yau structures'' in {\em loc. cit.}) induce proper Calabi-Yau structures (``right Calabi-Yau structures'' in {\em loc. cit.}) on any subcategory of (right respectively left) proper objects, and in particular a perfect pairing on endormorphisms of any such object.
We will recall this below and further show this induced pairing is compatible (fits into a snake relation with) with the algebraic open-string copairing defined above:

\begin{prop}\label{properobjectcopairingalgebra}
    Let $\cc$ be a (dg/$\ainf$) category equipped with a weak smooth Calabi-Yau structure $\sigma \in \r{HH}_{-n}(\cc)$, and let $c^{\sigma}$ be the $\sigma$-induced open-string copairing. If $\pp \subset \cc$ denotes any collection of right-proper objects, then for any $K,L$ in $\pp$, $c^{\sigma}_{K,L}$ is perfect. The same conclusion holds if $\pp \subset \cc$ is any collection of left-proper objects.
\end{prop}

\begin{proof}[Proof of Proposition \ref{properobjectcopairingalgebra}]

\cite{bravdyckerhoff}*{Thm 3.1}
shows in particular that any collection of right proper objects $\pp$ (in this case $\pp = \{L\}$) inherits (from the weak smooth Calabi-Yau structure $\sigma$ on $\cc$) a weak proper Calabi-Yau structure (the main result cited concerns non-weak Calabi-Yau structures, but we only need the weaker version, where the crux of the proof lies). Recall that a {\em weak proper Calabi-Yau structure} on a category $\pp$ is a map $tr_{\pp}: \r{HH}_*(\pp) \to \K[-n]$ such that under the natural identification $\r{HH}_{-n}(\pp)^{\vee} \to \hom_{\pp\!-\!\pp}(\pp_{\Delta}[n], \pp^\vee)$ (where $\pp^\vee:= \hom_{\pp}(-,-)^\vee$ denotes the {\em (linear) dual} of the diagonal bimodule), $tr_{\pp}$ corresponds to a bimodule quasi-isomorphism. Equivalently, the map $p^{tr_{\pp}}_{K,L}: H^*(\hom_{\pp}(K,L)) \otimes H^*(\hom_{\pp}(L,K)) \to \K[-n]$ induced by composing
$H^*(\hom_{\pp}(X,Y)) \times H^*(\hom_{\pp}(Y,X)) \stackrel{[\mu^2]}{\to} H^*(\hom(Y,Y)) \to \r{HH}_{*}(\pp) \stackrel{tr_{\pp}}{\to} \K[-n]$
    perfect pairing; in particular a weak proper Calabi-Yau structure has an associated perfect pairing $p^{tr_{\pp}}_{K,L}$. To define the weak proper Calabi-Yau structure on a collection of right proper objects $\pp$ in a weak smooth Calabi-Yau category $(\cc, \sigma)$,
    consider the Shklyarov pairing $\r{HH}_*(\cc) \otimes \r{HH}_*(\pp) \to \K$ associated to the hom pairing  $\cc^{op} \times \pp \to perf(\K)$ (this uses the fact that $\pp$ consists of right proper objects), then define $tr_{\pp}:= \langle \sigma, - \rangle$. As it induces a weak proper Calabi-Yau structure, $\sigma$ therefore induces a pairing on objects of $\pp$, defined as
    $p^{\sigma}_{K,L}:= p^{\langle \sigma, - \rangle}_{K,L}$.
    The argument in {\em loc. cit.} was written in the setting of dg categories but the same argument with the same proof carries through in the setting of $\ainf$ categories; or we can simply reduce to the dg case by replacing each $\ainf$ category with a quasi-equivalent dg category (noting that the notions of weak smooth Calabi-Yau structure, associated copairing, and proper object all transfer over). 

    It remains to check the snake relation between the algebraic open-string copairing $c^{\sigma}_{K,L}$ and $p^{\sigma}_{K,L}$.  Equivalently we will show $(c^{\sigma}_{K,L})^*: H^*(\hom_{\pp}(L,K))^{\vee}[-n] \to H^*(\hom_{\pp}(K,L)) $ and $(p^{\sigma}_{K,L})^*:  H^*(\hom_{\pp}(K,L)) \to H^*(\hom_{\pp}(L,K))^{\vee}[-n]$ are inverse for any $K,L \in \pp$. To do so, we will
first recall that in the proof from {\em loc. cit.} the following commutative diagram appears:
\begin{equation}
    \xymatrix{
        \r{HH}_{-n}(\cc) \ar[d]^{\cong} \ar[r] & \r{HH}_{n}(\pp)^{\vee} \ar[d]^{\cong}\\
        \hom_{\cc\!-\!\cc}(\cc^![n], \cc_{\Delta}) \ar[r]^{\Phi} & \hom_{\pp\!-\!\pp}(\pp_{\Delta}[n], \pp^\vee)
}
\end{equation}
where the top horizontal arrow is the map $\alpha \mapsto \langle \alpha, - \rangle$ using the Shklyarov-type pairing between $\cc$ and $\pp$ defined above. The vertical arrows are natural identifications (the left arrow for instance sends $\alpha$ to $ev(\alpha)$) and the bottom horizontal arrow, which we will expand upon, sends bimodule quasi-isomorphisms to bimodule quasi-isomorphisms (as it is induced by a composition of various functors between bimodule categories). It follows that the image of $ev({\sigma})$ induces a bimodule quasi-isomorphism $p_{\sigma}^*: \pp \stackrel{\sim}{\to} \pp^\vee[-n]$ (which implies by the commutativity that $tr_{\pp}:= \langle \sigma, - \rangle$ is a weak proper Calabi-Yau structure). In particular, by evaluating at the pair of objects $K,L \in \pp$, we obtain a quasi-isomorphism of chain complexes $(p_{\sigma})_{K,L}^*: \hom(K,L) \to \hom(L,K)^\vee[-n]$, which is (passing to the adjoint $\hom(K,L) \otimes \hom(L,K) \to \K[-n]$) the desired pairing by definition.

If we now apply the same functorial map $\Phi$ to the (homologically) inverse quasi-isomorphism of bimodules $CY = ev({\sigma})^{-1} \in \hom_{\cc\!-\!\cc}(\cc_{\Delta}, \cc^![n])$, and then evaluate at the pair of objects $K,L$, we will obtain a (homologically) inverse quasi-isomorphism of chain complexes $((p_{\sigma})_{K,L}^*)^{-1}: \hom(L,K)^\vee[-n] \to \hom(K,L)$. It suffices to show  that this inverse $((p_{\sigma})_{K,L}^*)^{-1}$ (which is induced by $CY = ev({\sigma})^{-1}$) can be computed as  $\left((ev({\sigma})^{-1})_{K,K}([id_K])_{L,L}\right)([id_L])^* = (c_{K,L}^{\sigma})^* = (- \otimes id)(c_{K,L}^{\sigma})$. This is an immediate consequence of verifying more generally that the following diagram is cohomologically commutative:
\begin{equation}\label{maindiagram}
\xymatrix{
    \hom_{\cc\!-\!\cc}(\cc, \cc^![n]) \ar[d]^{f \mapsto (f_{L,L}(id_L))_{K,K}(id_L)} \ar[r]^{\Phi} & \hom_{\pp\!-\!\pp}(\pp^\vee, \pp[n]) \ar[d]^{(ev_{K,L})_*}\\
    \hom_{\cc}(K,L) \otimes \hom_{\cc}(L,K) \ar[r]^{*\ \ \ \ \ \ \ } & \hom_{\K}(\hom_{\pp}(K,L)^\vee, \hom_{\pp}(L,K)[n])
}
\end{equation}
where the left vertical arrow takes a map $f: \cc \to \cc^![n]$ and looks at its copairing shadow, the bottom horizontal arrow takes $c \mapsto (c^*: \phi \mapsto (\phi \otimes id) \circ c)$, and the right vertical arrow simply considers the underlying map of chain complexes associated to an element $\hom_{\pp\!-\!\pp}(\pp^\vee, \pp[n])$ for a pair $(K,L)$ both in $\pp$.

To verify this, let us describe the map $\Phi$ from {\em loc. cit.} in some more detail.
First, one can restrict a $\cc\!-\!\cc$ bimodule along the inclusion $i: \pp \to \cc$ on the right; we'll call the right restriction of a bimodule $\mc{B}$ $\mc{B}|_{\pp}$ following \cite{bravdyckerhoff}. Restriction is functorial, inducing a map
\begin{equation}\label{restrictiontop}
    res_{\pp}: \hom_{\cc\!-\!\cc}(\cc, \cc^![n]) \to \hom_{\cc\!-\!\pp}(\cc|_{\pp}, \cc^![n]|_{\pp});
\end{equation}
the image of $CY = ev(\sigma)^{-1}$ is simply $CY|_\pp$, and evidently for $K,L$ in $\pp$, this restriction commutes with the natural maps that take $CY$ respectively $CY|_{\mc{P}}$ to their copairing shadow for a pair of objects in $\pp$. 

We now apply the functor $\hom_{\cc}(-, \cc)$ (where $\hom_{\cc}$ is taken in the category of left modules), a contravariant functor from $\cc\!-\!\pp$ bimodules to $\pp\!-\!\cc$ bimodules (note the right module structure on the target $\cc$ survives as does the right $\pp$ module structure on the source which becomes a left module structure by dualization):
\begin{equation}\label{dualize}
    \Phi_1: \hom_{\cc\!-\!\pp}(\cc|_{\pp}, \cc^![n]|_{\pp}) \to \hom_{\pp\!-\!\cc}(\hom_{\cc}(\cc^!|_{\pp}[n], \cc), \hom_{\cc}(\cc|_{\pp}, \cc));
\end{equation}
On the level of chain complexes, for an object $K \in \pp$
\begin{equation}
    \begin{split}
(\Phi_1)_{K,K}:= \hom_{\K}( \cc(K,K), \cc^![n](K,K)) &\to \hom_{\K}(\hom_{\K}(\cc^!(K,K), \cc(K,K)), \hom_{\K}(\cc|_{\pp}(K,K), \cc(K,K)))\\
f &\mapsto (-) \circ f
\end{split}
\end{equation}
Next we observe that by properness of the bimodule $\cc|_{\pp}$ (as $\pp$ consists of right-proper objects), there is a natural isomorphism of bimodules $ev: \cc|_{\pp} \stackrel{\cong}{\to} (\cc|_{\pp}^\vee)^\vee$. Composing with $ev$, we see that $\cc^![n]|_{\pp} = \hom_{\cc\!-\!\cc}(\cc, \cc|_{\pp} \otimes \cc)[n] \stackrel{ev \otimes id \circ -}{\longrightarrow} \hom_{\cc\!-\!\cc}(\cc, (\cc|_{\pp}^\vee)^\vee \otimes \cc)[n]$. Using the canonical map $\alpha: M^\vee \otimes_{\K} N \to \hom_{\K}(M,N)$ which exists for any $M$, $N$ (modules, chain complexes etc), which is an isomorphism if $M$ is proper (applicable here because $\cc|_{\pp}$ is proper), we can map this latter complex isomorphically to
    $\stackrel{\cong}{\to} \hom_{\cc\!-\!\cc}(\cc, \hom_{\K}((\cc_{\pp}^\vee), \cc))[n]$, which by hom-tensor adjunction is equal to (in fact coincides on the level of chain complexes built from the bar construction)  $\hom_{\cc}(\cc \otimes_{\cc} \cc_{\pp}^\vee, \cc)[n]$.
    Then there is a natural ``collapse'' quasi-isomorphism of bimodules 
    \begin{equation}\label{collapse}
        \cc \otimes_{\cc} \cc|_{\pp}^\vee \stackrel{\cong}{\to} \cc|_{\pp}^\vee;
    \end{equation} 
    On the level of chain complexes, the homologically inverse chain map $\cc_{\pp}^\vee(K,L) \to (\cc \otimes_{\cc} \cc_{\pp}^\vee)(K,L)$ sends $x \mapsto id_K \otimes x \in \cc(K,K) \otimes_K \cc|_{\pp}^\vee(K,L) \subset (\cc \otimes_{\cc} \cc|_{\pp}^\vee)(K,L)$ (using the bar model). Composing with \eqref{collapse} induces a quasi-equivalence $\hom_{\cc}(\cc|_{\pp}^\vee, \cc) \stackrel{\cong}{\to} \hom_{\cc}(\cc \otimes_{\cc} \cc|_{\pp}^\vee, \cc)$; inverting this we all together get a quasi-isomorphism $\eta: \cc^![n]|_{\pp} \cong \hom_{\cc}(\cc|_{\pp}^\vee, \cc)$; which on the level of chain complexes (by above) admits the following description 
    \begin{equation}
        \begin{split}
        \eta_{K,K}: \hom_{\cc\!-\!\cc}(\cc, \cc|_{\pp}(-,K) \otimes \cc(K,-))[n] &\to \hom_{\cc}(\cc|_{\pp}^\vee(-,K), \cc(-, K))\\
        g \mapsto (\phi \mapsto \left(\phi \otimes \id\right) (g(id_K)))
    \end{split}
    \end{equation}
The quasi-isomorphism $\eta$ induces
\begin{equation}
    \begin{split}
    \Phi_2: \hom_{\pp\!-\!\cc}(\hom_{\cc}(\cc^!|_{\pp}[n], \cc), \hom_{\cc}(\cc|_{\pp}, \cc)) &\stackrel{\cong}{\to} \hom_{\pp\!-\!\cc}( \hom_{\cc}(\hom_{\cc}(\cc|_{\pp}^\vee, \cc)[n], \cc), \hom_{\cc}(\cc|_{\pp}, \cc))\\
\end{split}
\end{equation}
which on the level of chain complexes sends
$h \mapsto (\psi \mapsto h(\psi \circ \eta))$.
Finally there is a canonical morphism from $\cc|_{\pp}^\vee$ to its double left-$\cc$-module dual $\hom_{\cc}(\hom_{\cc}(\cc|_{\pp}^\vee, \cc), \cc)$ which is an isomorphism along any object $p \in \pp$ for which $\cc|_{\pp}^\vee(-, p)$ is a perfect $\cc$-module; this holds in our case for all $p \in \pp$ seeing as $\cc|_{\pp}^\vee(-,p)$ is a proper $\cc$-module and $\cc$ is smooth (over smooth categories, proper modules are automatically perfect; compare \cite{GPSwrappedconstructible}*{Lemma A.8}):
\begin{equation}\label{doubledualdual}
    \cc|_{\pp}^\vee \stackrel{\cong}{\to} \hom_{\cc}(\hom_{\cc}(\cc|_{\pp}^\vee, \cc), \cc).
\end{equation}
On the chain level the above map sends $\phi \mapsto (z \mapsto z(\phi))$.

Precomposing with the inverse of \eqref{doubledualdual} (shifted by $-n$)  we obtain an equivalence
\begin{equation}
    \Phi_3: \hom_{\pp\!-\!\cc}( \hom_{\cc}(\hom_{\cc}(\cc|_{\pp}^\vee, \cc)[n], \cc), \hom_{\cc}(\cc|_{\pp}, \cc))  \to \hom_{\pp\!-\!\cc}( \cc|_{\pp}^\vee[-n], \hom_{\cc}(\cc|_{\pp}, \cc))
\end{equation}
Compose once more with restriction on the right to $\pp$:
\begin{equation}
    res_{\pp}: \hom_{\pp\!-\!\cc}( \cc|_{\pp}^\vee[-n], \hom_{\cc}(\cc|_{\pp}, \cc)) \to \hom_{\pp\!-\!\pp}(\pp^\vee, \hom_{\cc}(\cc|_{\pp}, \cc|_{\pp})).
\end{equation}
Finally as $\pp \subset \cc$ there is a Yoneda isomorphism $ \pp \stackrel{\cong}{\to} \hom_{\cc}(\cc|_{\pp}, \cc|_{\pp})$.
On the level of chain complexes given $K,L \in \pp$ a homological inverse $\hom_{\cc}(\cc|_{\pp}(-,L), \cc|_{\pp}(-,K))$ can be described by sending $f \mapsto f(id_L)$.  
Composing with the inverse to the Yoneda map on the right we get
\begin{equation}
    \Phi_4: \hom_{\pp\!-\!\pp}(\pp^\vee, \hom_{\cc}(\cc|_{\pp}, \cc|_{\pp})) \to \hom_{\pp\!-\!\pp}(\pp^\vee, \pp[n]).
\end{equation}
The map $\Phi$ can now be described as $\Phi_4 \circ res_{\pp} \circ \Phi_3 \circ \Phi_2 \circ \Phi_1 \circ res_{\pp}$. Assembling all the chain level descriptions together, given an element $f_{K,K} \in \hom_{\K}(\cc(K,K), \cc^!(K,K))$ and for objects $K,L \in \pp \subset \cc$ and an element $\phi \in \pp^\vee(K,L)$, $\Phi(f)(\phi) \in \pp(L,K)$ can be described up to chain equivalence as follows (using the above chain level descriptions of homological inverses to certain maps, and ignoring the $res_{\pp}$ terms as we've already restricted to objects of $\pp$):
\begin{equation}
    \begin{split}
        f &\stackrel{\Phi_1}{\mapsto} h := (-) \circ f \stackrel{\Phi_2}{\mapsto} (\psi \mapsto h(\psi \circ \eta) = \psi \circ \eta \circ f ) \\
        & \stackrel{\Phi_3}{\mapsto}  \phi \mapsto (ev_{\phi} \circ \eta_{K,K} \circ f_{L,L}) \\
        &\stackrel{\Phi_4}{\mapsto}   \phi \mapsto (ev_{\phi} \circ \eta_{K,K} \circ f_{L,L}(id_L))\\
&= \phi \mapsto ev_{\phi} \circ (x \mapsto x \otimes id \circ (f_{L,L}(id_L))_{K,K}(id_K))\\
&= \phi \mapsto (\phi \otimes id) ((f_{L,L}(id_L))_{K,K}(id_K)).\\
&= ((f_{L,L}(id_L))_{K,K}(id_K))^*
    \end{split}
\end{equation}
This verifies \eqref{maindiagram} so we are done in the right proper case.

The proof for the left proper case is essentially identical to the right proper case (one simply swaps the order of the pairing between $\cc$ and $\pp$ and similarly all bimodule structures in what follow, or alternatively think about right properness of $\pp^{op}$ in $\cc^{op}$, which also carries a weak smooth Calabi-Yau structure by $\r{HH}_*(\cc) \cong \r{HH}_*(\cc^{op})$).
\end{proof}

\subsection{A criterion for non-properness}\label{subsec:openness}

Given a graded unital algebra $A$, we will say a copairing $c \in A \otimes A[k]$ is {\em entirely degenerate} if the induced map $c^*: A^{\vee}[-k] \to A$ has nilpotent image. It is easy to see that an entirely degenerate $c$ can only be non-degenerate if $1$ is nilpotent i.e., if $A=0$.

On a weak smooth Calabi-Yau category $\cc$, in order to talk about the degeneracy properties of its algebraic closed-string copairing, 
we transfer the copairing to Hochschild cohomology, a unital algebra, via the Calabi-Yau isomorphism $\r{HH}_{*-n}(\cc) \cong \r{HH}^*(\cc)$ (or one can equivalently transfer the ring structure to $\r{HH}_*(\cc)$).
The following terminology is motivated by the usage of ``open manifold'' to describe manifolds all of whose components are non-compact:
\begin{defn}\label{def:algebraicallyopen}
    A (weak) smooth Calabi-Yau category $\cc:=(\cc, \sigma)$ is {\em algebraically open} if its algebraic closed-string copairing, thought of as living on $HH^*(\cc)$ via the Calabi-Yau isomorphism, is entirely degenerate. 

    Similarly, we say an object $X$ in a (weak) smooth Calabi-Yau category $\cc$ is {\em algebraically open} if its algebraic open-string copairing on $H^*(\hom_{\cc}(X,X))$ is entirely degenerate.
\end{defn}

An immediate Corollary of the previous section is that algebraic openness provides a non-properness or vanishing criterion:
\begin{cor}\label{opennonproperzero}
    Let $\cc:=(\cc, \sigma)$ be an algebraically open weak smooth Calabi-Yau category. Then $\cc$ is either non-proper or zero. Similarly, if $X \in \cc$ is an algebraically open object in a weak smooth Calabi-Yau category, $X$ is either (both left and right) non-proper or zero.
\end{cor}
\begin{proof}
    Suppose such a $\cc$ is proper. Then Proposition \ref{shklyarovprop} implies that its algebraic closed-string copairing $c_{alg}$ is non-degenerate, but it is also entirely degenerate, hence $\r{HH}^*(\cc)$ is zero, which implies $\cc$ is zero (as the cohomological endormophisms of any object are a unital module over $\r{HH}^*(\cc)$).
    Similarly, suppose $X\in \cc$ is right proper. Proposition \ref{properobjectcopairingalgebra} implies the algebraic open string copairing is non-degenerate, but by algebraic openness this can only happen if $X = 0$. The left proper case is the same.
\end{proof}

We conclude this section by showing that the categorical and object-wise algebraic openness conditions are sufficiently strong so as to persist under quotients/homological epimorphisms and idempotent summands.
\begin{prop}\label{quotientpreservesopenness}
    Let $\cc$ be (weak) smooth Calabi-Yau, and suppose $\cc$ is algebraically open. Then any quotient/homological epimorphism $\dd$ of $\cc$ inherits from $\cc$ a (weak) smooth Calabi-Yau structure, with respect to which it is also  algebraically open.
\end{prop}
\begin{proof}
The algebraically open condition is equivalent to the map $CY_{\cc} \circ c^*_{alg}: \r{HH}_*(\cc)^{\vee} \to \r{HH}^{*+n}(\cc)$ having nilpotent image, so we will show this condition transfers along quotients/homological epimorphisms.
    Let $f: \cc \to \dd$ denote the quotient map (or homological epimorphism).
        Lemmas \ref{copairinglocalization} and \ref{localizeCY} imply that $\dd$ is again smooth, that the algebraic closed-string copairing on $\dd$ factors through the one on $\cc$ (i.e., $(f_{*} \otimes f_*) (c_{alg}^{\cc}) = c_{alg}^{\dd})$, that $\dd$ inherits a weak smooth Calabi-Yau structure from the one on $\cc$, and that there is a commutative diagram (using the inherited Calabi-Yau structure to define $CY_{\dd}$)
    \[
        \xymatrix{ \r{HH}_*(\cc) \ar[r]^{f_*} \ar[d]^{CY_{\cc}}_{\cong} &  \r{HH}_*(\dd) \ar[d]^{CY_{\dd}}_{\cong} \\
    \r{HH}^*(\cc) \ar[r]^{f_{\sharp}} & \r{HH}^*(\dd) 
    } \]
    It follows that $CY_{\dd} \circ (c_{alg}^{\dd})^*$ therefore factors as $f_*^{\vee}$ followed by $CY_{\cc} \circ (c_{alg}^{\cc})^*$ (whose image is nilpotent by hypothesis) followed by the map $f_{\sharp}$ (an algebra map by Lemma \ref{localizehh}); hence it has nilpotent image.
\end{proof}

\begin{lem}\label{quotientpreservesopennessobject}
    Let $\{c_{L,K}^f\}_{K,L \in perf(\cc)}$ be the collection of copairing shadows  associated to any closed morphism of bimodules $f: \cc_{\Delta} \to \cc^![n]$ above. If $L$ is algebraically open with respect to $c_L^f$, and $K$ is any idempotent summand of $L$ in $perf(\cc)$, then $K$ is algebraically open with respect to $c_K^f$.

In particular, with respect to the collection of algebraic open-string copairings  $\{c_{L,K}^{\sigma}\}_{K,L \in perf(\cc)}$ induced by a weak smooth Calabi-Yau structure, the condition of algebraic openness is inherited by idempotent summands.
\end{lem}
\begin{proof}
    By Morita invariance we replace $\cc$ with $perf(\cc)$ and think of $f$ as a morphism of bimodules over $perf(\cc)$.
    Let $[p]: L \to K$ and $[i]: K \to L$ be the homological projections of $L$ onto $K$ and inclusion of $K$ into $L$, so $[i] \circ [p]: L \to L$ is an idempotent, and $[p] \circ [i] = [id_K]$. Composition ($[\mu^2]$) with $[p]$ on the right ($[\mu^2]([p],-)$) and precomposition with $[i]$ on the left ($[\mu^2](-,[i])$)  gives a map $\pi_{L,K}: H^*\hom_{\cc}(L,L) \to  H^*\hom_{\cc}(K,K)$ and hence a map $\pi_{L,K} \otimes \pi_{L,K}: H^*\hom_{\cc}(L,L) \otimes H^*\hom_{\cc}(L,L) \to H^*\hom_{\cc}^*(K,K) \otimes H^*\hom_{\cc}(K,K)$. The main claim that this map sends $c_L$ to $c_K$. Supposing it did, the proof can be completed as follows: since $\pi_{L,K}\otimes \pi_{L,K}$ sends $c_L$ to $c_K$, it follows that $c_K^*$ factors as $H^*\hom_{\cc}(K,K)^{\vee}[-n] \stackrel{\pi_{L,K}^*}{\to} H^*\hom_{\cc}(L,L)^{\vee}[-n] \stackrel{c_L^*}{\to} H^*\hom_{\cc}(L,L) \stackrel{\pi_{L,K}}{\to} H^*\hom_{\cc}(K,K)$. Since the last map $\pi_{L,K}$ is an algebra map and hence sends nilpotent elements to nilpotent elements, we are done.

    Finally, the claim that the map $\pi_{L,K}$ sends $c_L$ to $c_K$ is an elementary consequence of the map $f$ being a closed morphism of $\ainf$ bimodules. To spell this out, it is convenient to recast the induced map on chain complexes $f_{A,B}(-)_{C,D}$ (by hom-tensor adjunction) as a map $f_{A,B,C,D}: H^*(\hom(A,B)) \otimes H^*(\hom(C,D)) \to H^*(\hom(A,D)) \otimes H^*(\hom(C,B))$ for every tuple of objects $A$,$B$,$C$,$D$, compatible with the homology compositions on the right and left (since $f$ is a closed morphism). By definition, $c_L: = f_{L,L,L,L}([id_L] \otimes [id_L])$, and $\pi_{L,K} \otimes \pi_{L,K} = (\mu^2(-, [i]) \circ \mu^2([p], -))^{\otimes 2}$, Hence, iteratively appealing to the compatibility with multiplication (using the abuse of notation $\mu^2$ for $[\mu^2]$), we eventually learn that $\pi_{L,K} \otimes \pi_{L,K} \circ f_{L,L,L,L}([id_L] \otimes [id_L]) = f_{K,K,K,K}(\mu^2([i], \mu^2([id_L], [p])) \otimes \mu^2([i], \mu^2([id_L], [p]))) = f_{K,K,K,K}(\mu^2([i],[p]) \otimes \mu^2([i], [p])) = f_{K,K,K,K}([id_K] \otimes [id_K]) = c_K$.
\end{proof}
Similarly, algebraic openness of an object is preserved by quotients/homological epiomorphisms:
\begin{lem}\label{openobjectquotient}
    If $(\cc, \sigma)$ is a weak smooth Calabi-Yau category and let $f: \cc \to \dd$ any quotient/homological epimorphism.  If $X \in \cc$ is algebraically open, then $fX \in \dd$ is algebraically open (with respect to the weak smooth Calabi-Yau structure $f_*\sigma$ induced by the map from $\cc$ by Lemma \ref{localizeCY}).
\end{lem}
\begin{proof}
    Lemma \ref{opencopairingquotient} implies that $f$ sends the open-string copairing on $X \in (\cc, \sigma)$ to the one on $fX \in (\dd, f_*\sigma)$, after which the proof follows the same lines as the start of the previous Lemma.
\end{proof}

\section{Comparing algebraic and geometric copairings}\label{sec:geometrytoalgebra}
Returning to symplectic geometry, let $X$ be a non-degenerate Liouville manifold and $\w(X)$ its wrapped Fukaya category. We review (and in a couple cases spell out) a series of results from \cite{ganatra1_arxiv, rezchikov} that allow one to compare the geometric copairings studied in \S \ref{sec:degeneracylem} and the algebraic copairings studied in \S \ref{sec:copairingcy}.

Let us recall that there is a geometrically defined {\em open-closed map}
\[
    \oc: \r{HH}_{*-n}(\w(X)) \to SH^*(X),
\]
(as say defined in \cite{abouzaid_generation}), as well as a {\em closed-open map} (see e.g., \cite{seidelicm, ganatra1_arxiv})
\[
    \co: SH^*(X) \to \r{HH}^*(\w(X)).
\]
The first series of results we'll need are:
\begin{thm}[\cite{ganatra1_arxiv}, see also \cite{Gcircleactions} Thm. 3]\label{ociso}
    If $X$ is non-degenerate, then $\oc$ and $\co$ are isomorphisms, $\co$ is a unital ring map, $\w(X)$ is smooth and the element $\sigma_{\oc} \in \r{HH}_{-n}(\w(X))$ given by $\oc^{-1}(1)$ is a weak smooth Calabi-Yau structure. Furthermore the composed isomorphism $\co \circ \oc: \r{HH}_{*-n}(\w(X)) \to \r{HH}^*(\w(X))$ is homologically equal to the inverse to capping with $\sigma_{\oc}$, which we have denoted $CY_{\sigma_{\oc}}$.
\end{thm}
Theorem \ref{ociso} allows one to situate $\w(X)$ within the setting of \S
\ref{sec:copairingcy} and in particular equip it with algebraic open and
closed-string copairings. The comparison of closed-string copairings we need is:
\begin{thm}[Rezchikov, in preparation \cite{rezchikov}]\label{occopairing}
    If $X$ is non-degenerate, then the map $\oc$ sends the algebraic
    closed-string copairing $c_{alg}$ defined on $\w(X)$ (as above) to the geometric copairing $c_{SH}$ on
    $SH^*(X)$.
\end{thm}
A corollary of Theorem \ref{ociso} and Theorem \ref{occopairing} is:
\begin{cor}\label{cor:algebracopairing}
    Using the Calabi-Yau isomorphism $CY_{\sigma_{\oc}}:\r{HH}_{*-n}(\w(X)) \to \r{HH}^*(\w(X))$ to think of Hochschild cohomology as a unital algebra with copairing, $\co$ is an isomorphism of unital algebras with copairing.
\end{cor}
\begin{proof}
    $\co$ is a already a unital algebra isomorphism, so we just need to observe that the algebraic copairing on Hochschild cohomology, by definition $CY_{\sigma_{\oc}}^{\otimes 2} (c_{alg})$ is by Theorem \ref{ociso} equal to $(\co \circ \oc)^{\otimes 2}(c_{alg})$ which by Theorem \ref{occopairing} is $\co^{\otimes 2}(c_{SH})$ as desired.
\end{proof}

Next we turn to comparing the (algebraic and geometric) open-string copairings on $(\w(X), \sigma_{\oc})$. The relevant comparison is implicit in \cite{ganatra1_arxiv} as we now explain. In {\em loc. cit.} a geometric morphism of bimodules $\mc{CY}: \w(X) \to \w(X)^![n]$ was constructed (called in {\em loc. cit.} the ``non-compact Calabi-Yau morphism''), by counting discs with two outputs and arbitrarily many inputs in between, with one distinguished input on each component of the boundary minus outputs (fixing the cross ratio of the 4 special points, the two distinguished inputs and the two outputs). Like the open-string copairing, this morphism arises from counting discs with two outputs; in particular
\begin{lem}\label{cyshadow}
    The geometric open-string copairing \eqref{opencopairing} is the copairing shadow $c_{K,L}^{\mc{CY}}$ of $\mc{CY}$ in the sense defined in \S \ref{subsec:copairings}.
\end{lem}
\begin{proof}
    The copairing shadow $c^{\mc{CY}}_{K,L}$ of the morphism $\mc{CY}$ for a pair of objects $K,L$ is by definition 
    $(\mc{CY}_{K,K})(id_K))_{L,L}(id_L)$.
    This corresponds to counts of discs with two inputs and two outputs arranged in alternating fashion (the lowest order term in $\mc{CY}$), with intputs given by homological units for $K$ and $L$ (given by a count of unstable discs with one output for each Lagrangian as e.g., described in {\em loc. cit.}). By a standard gluing argument, this is chain homotopy equivalent to the geometric copairing as desired.
\end{proof}

\begin{prop}\label{copairingCY}
    Let $X$ be a non-degenerate Liouville manifold and let $(\mc{W}(X), \sigma_{\oc}:=\oc^{-1}(1))$ be its wrapped Fukaya category with its geometric weak smooth Calabi-Yau structure as in Theorem \ref{ociso}. 
    Then, for any $K,L \in \w(X)$, the geometric open-string copairing \eqref{opencopairing} is equal to the algebraic open-string copairing $c_{K,L}^{\sigma_{\oc}}$ induced by the element $\sigma_{\oc}$.
\end{prop}
\begin{proof}
    We recall that in \cite{ganatra1_arxiv} the geometric map $\mc{CY}: \w(X) \to
    \w(X)^![n]$ described above was proven, under the given non-degeneracy
    hypotheses, to be inverse to the cap product isomorphism $ev(\sigma_{\oc})$
    induced by the weak Calabi-Yau structure $\sigma_{\oc}$ (this is slightly
    implicit in {\em loc. cit.}, but spelled out in \cite{Gcircleactions}*{Thm.
3}). Hence the $\sigma_{\oc}$-induced algebraic open-string copairing is by
definition the copairing shadow of the map $\mc{CY}$, which Lemma
\ref{cyshadow} verifies is
    precisely the open-string copairing \eqref{opencopairing}. 
\end{proof}

\section{Proof of the Main Theorem}\label{sec:proof}

The proof of the Main Theorem follows immediately from the above results, for instance (1) amounts to the incompatibility of having a simultaneously smooth, proper, and non-vanishing non-degenerate wrapped Fukaya category (whose algebraic hence geometric copairing must be non-degenerate) with the known degeneracy property of the geometric copairing.  In all cases the arguments proceed by showing properness implies vanishing.
\begin{proof}[Proof of Main Theorem]
    Assume $\w(X)$ is non-degenerate as hypothesized. Starting with (1), suppose $\w(X)$ is furthermore proper; we need to show that then $\w(X) = 0$. By Proposition \ref{shklyarovprop} properness implies the algebraic closed string copairing on $\r{HH}_{*-n}(\w(X))$ is non-degenerate, which implies by Theorem \ref{occopairing} that the geometric copairing $c_{SH}$ is non-degenerate. Since the degeneracy Lemma (Theorem \ref{degeneratecopairing}) establishes the geometric copairing is entirely degenerate, this can only happen if $SH^*(X) = 0$ which implies $\w(X) = 0$ by a standard argument (each $HW^*(L,L)$ is a unital module over $SH^*(X)$) as desired.
    To generalize to (1'), note that the degeneracy Lemma (Theorem \ref{degeneratecopairing}) implies in light of Corollary \ref{cor:algebracopairing} 
    that $(\w(X), \sigma_{\oc})$ is {\em algebraically open} in the sense of Definition \ref{def:algebraicallyopen}, a property which
    by Proposition \ref{quotientpreservesopenness} 
    is inherited by quotients or homological epimorphisms, and which by Corollary \ref{opennonproperzero} also implies non-properness or vanishing.

    For (2), let's assume without loss of generality that a given open exact Lagrangian $L$ is right proper (the left proper case is the same, or can be handled by consider $L$ as a right proper object in $\w(X^-) = \w(X)^{op}$ which is also non-degenerate). Propositions \ref{copairingCY} and \ref{properobjectcopairingalgebra} therefore imply its geometric copairing is non-degenerate, whereas Theorem \ref{degeneratecopairingopenstring} implies its geometric copairing is entirely degenerate, from which it follows that $L=0$. For (2'), we think of Theorem \ref{degeneratecopairingopenstring} and Proposition \ref{copairingCY} as implying that $L$ is algebraically open in the sense of Definition \ref{def:algebraicallyopen} a property which is both inherited by idempotent summands (by Proposition \ref{quotientpreservesopennessobject}) and also implies by Corollary \ref{opennonproperzero} non-properness or vanishing.
\end{proof}

Turning towards Question \ref{infiniteorzeroSH} (and Question \ref{infiniteorzero})  we note that the same argument would not imply that if $SH^*(X)$ was finite dimensional, it was zero: the point is that $SH^*(X)$ does not inherit a canonical pairing when it is finite dimensional, but very notably the Hochschild homology of a proper category does (which plays a crucial role in establishing non-degeneracy of the algebraic copairing in Proposition \ref{shklyarovprop}, and hence the geometric copairing by Theorem \ref{occopairing}). Similarly, $HW^*(L,L)$ does not inherit a pairing if it is finite-dimensional but it does if $L$ is left or right proper. That being said, we can articulate the following alternatives which are immediate consequences of the Main Theorem (and the open-closed isomorphisms):
\begin{cor} \label{zeroorextalternative}
    Given a non-degenerate Liouville (e.g., a Weinstein) manifold $X$, and an open exact Lagrangian $L \subset X$,
    \begin{enumerate}
\item[(1)]$SH^*(X)$ is either zero or equal to the Hochschild cohomology of a smooth, non-proper (and ``algebraically open'') Calabi-Yau category $\w(X)$. 
\item[(2)] $HW^*(L,L)$ is either zero or equal to the self-Ext of a perfect, non-proper (and ``algebraically open'') module over a non-proper Calabi-Yau category $\w(X)$.
    \end{enumerate}
\end{cor}
\begin{rem}\label{zeroorexactcounterexample}
    It's interesting to ask to what degree the above Corollary can shed further light on Questions \ref{infiniteorzeroSH} and \ref{infiniteorzero}. Smooth non-proper non-zero categories with finite Hochschild cohomology appear to be quite scarce in the literature, but do exist, which might suggest what to look for in an attempt to answer Question \ref{infiniteorzeroSH} in the negative.
    For instance, \cite{finitehochschild} gives an example of a proper non-smooth associative algebra $A$ with $\r{HH}^*(A)$ finite non-zero. By \cite{elaginluntsschnurer} the category of proper dg modules over such an $A$ is smooth non-proper and gives (taking the endormophisms of a (split)-generator) an example of a smooth non-proper non-zero algebra $B$ with $\r{HH}^*(B)$ finite non-zero (for general reasons $\r{HH}^*(B) = \r{HH}^*(A)$). In characteristic zero,  Weyl algebras give examples of smooth non-proper Calabi-Yau algebras with finite non-zero Hochschild cohomology (see \cite{weylHH, weylCY}).
The author is grateful to Efimov and Yeung for conversations about these examples.
\end{rem}

\begin{rem}\label{rem:nonexact}
    If $X$ is a non-exact open manifold (e.g., say $X$ has contact-type boundary at infinity) and $L \subset X$ is an open Lagrangian submanifold, the Main Theorem  should hold whenever 
    \begin{enumerate}
        \item the wrapped Fukaya category $\w(X)$, symplectic cohomology $SH^*(X)$ are defined with their open-closed structures, the wrapped Fukaya category is non-degenerate in the sense above (meaning in particular one has access to a generating collection of Lagrangian submanifolds), and further satisfy all of the structure results from \S \ref{sec:geometrytoalgebra} (including e.g., that $\w(X)$ has a weak smooth Calabi-Yau structure).

        \item $SH^*(X)$ respectively $HW^*(L,L)$ satisfy the degeneracy Lemmas (Thms. \ref{degeneratecopairing} and \ref{degeneratecopairingopenstring}).
    \end{enumerate}
Remarks \ref{rem:degeneracylemmanonexact} and \ref{rem:degeneracylemmanonexactlagrangian} examine criteria under which 
Thms. \ref{degeneratecopairing} and \ref{degeneratecopairingopenstring}
could fail or succeed in such cases.
\end{rem}

\begin{rem}
    To some degree, the structures appearing in (at least Statements (1) and (1') of the) Main Theorem are categorically dual to some of the arguments of \cite{Gautogen_arxiv} (in the sense that the roles of the finiteness conditions `smooth'/`proper' inducing `copairings/pairings' and corresponding properties of functors `fully faithful'/`localization' are interchanged). In particular in {\em loc. cit.}, one obtains, seeing as compact Fukaya categories are always proper, strong constraints/implications from having a smooth fully faithfully embedded subcategory whereas here one obtains, seeing as wrapped Fukaya categories of non-degenerate Liouville manifolds are always smooth, strong constraints/implications from having a proper quotient category. The form of those constraints depend on the structure of the respective closed string group ($QH^*(X)$ respectively $SH^*(X)$).
\end{rem}

\appendix
\section{Localizing (weak) smooth Calabi-Yau structures}\label{app:smooth}
We recall some known (partly folk) properties of the behavior of smooth categories and
(weak) smooth Calabi-Yau structures 
under localization/quotient/homological epimorphism, drawing in part
from \cite{ganatra1_arxiv, Gcircleactions, bravdyckerhoff}. 

We continue with the notation from \S \ref{sec:copairingcy}.
Recall that an
($\ainf$) functor $f: \cc \to \dd$ induces a pushforward (or extension of
scalars) functor on bimodule categories $f_*: [\cc,\cc] \to [\dd,\dd]$, defined
by taking left and right one-sided tensor products with the graph of $f$.  Up
to isomorphism (meaning isomorphism in the cohomology category), this functor
sends representables (for $(K,L)$) to representables (for $(f(K), f(L))$).
There is also a restriction of bimodules functor $f^*:=(f,f)^*: [\dd,\dd] \to
[\cc,\cc]$. While neither $f_*$ or $f^*$ typically send the diagonal to the
diagonal, there are canonical morphisms $d: \cc_{\Delta} \to f^* \dd_{\Delta}$
and $c: f_*\cc_{\Delta} \to \dd_{\Delta}$.
\begin{alem}\label{localizationlem}[Compare \cite{efimov} Cor. 3.8 and Prop 3.4, \cite{GPSsectorsoc} Lemma 3.15]
    If $f: \cc \to \dd$ is a (dg or $\ainf$) quotient functor, then the canonical morphism $c: f_*(\cc_{\Delta}) \to \dd_{\Delta}$ is an isomorphism.
\end{alem}
Recall that the {\em quotient} of an $\ainf$ category $\cc$ is another $\ainf$ category $\dd$ equipped with a functor $\cc \to \dd$ which is in some sense initial among all functors sending a subcategory to zero; the definition and various explicit models are discussed in \cite{drinfeldDG,lyub1, lyub2} (the first reference in the dg case). We say $f$ is a {\em homological epimorphism} (compare \cite{efimov}) if $c: f_*\cc_{\Delta} \to \dd_{\Delta}$ is an isomorphism; Lemma \ref{localizationlem} ensures this is a weaker condition than asking $f$ be a quotient.

Recall that Hochschild homology is covariantly functorial, meaning that functors $f: \cc \to \dd$ induce pushforward maps $f_*: \r{HH}_*(\cc) \to \r{HH}_*(\dd)$.  
While Hochschild cohomology is not generally functorial, we do have covariant functoriality under (dg or $\ainf$) quotient functors:
\begin{alem}\label{localizehh}
    If $f: \cc \to \dd$ is any quotient/localization functor (or more generally a homological epimorphism), then there is an induced pushforward algebra map on Hochschild cohomology $f_{\sharp}: \r{HH}^*(\cc)\to \r{HH}^*(\dd)$, with respect to which the pushforward on Hochschild homology $f_*: \r{HH}_*(\cc) \to \r{HH}_*(\dd)$ is a map of $\r{HH}^*(\cc)$-modules.
\end{alem}
\begin{proof}[Proof Sketch]
    For the first statement, note $f$ induces a functor $f_*:  [\cc,\cc] \to [\dd,\dd]$ and hence a map from $\r{HH}^*(\cc)$ to (cohomological) endomorphisms of $f_* \cc_{\Delta}$. In light of Lemma \ref{localizationlem}, localization implies there is a quasi-isomorphism $c \circ (-) \circ c^{-1}: \hom_{\dd\!-\!\dd}(f_* \cc_{\Delta}, f_* \cc_{\Delta}) \to \hom_{\dd\!-\!\dd}(\dd_{\Delta}, \dd_{\Delta})$. The second assertion follows immediately from the (cohomologically) commutative diagram, for any $[\alpha] \in H^*(\cc_{\Delta} \otimes_{\cc\!-\!\cc} \cc_{\Delta}) = \r{HH}_*(\cc)$:
    \[
        \xymatrix{
            \cc_{\Delta} \otimes_{\cc\!-\!\cc} \cc_{\Delta}  \ar[d] \ar[r] &  f^*\dd_{\Delta} \otimes_{\cc\!-\!\cc} \cc_{\Delta}\ar[d] \ar[r] & 
            \dd_{\Delta}\otimes_{\cc\!-\!\cc} (f_* \cc_{\Delta}) \ar[d] \ar[r]^{c} & \dd_{\Delta} \otimes_{\dd\!-\!\dd} \dd_{\Delta} \ar[d]\\
        \cc_{\Delta} \otimes_{\cc\!-\!\cc} \cc_{\Delta} \ar[r] & f^*\dd_{\Delta} \otimes_{\cc\!-\!\cc}\cc_{\Delta} \ar[r]& \dd_{\Delta}\otimes_{\cc\!-\!\cc} (f_* \cc_{\Delta}) & \dd_{\Delta} \otimes_{\dd\!-\!\dd} \dd_{\Delta} \ar[l]^{c^{-1}}}
    \]
    where the first horizontal map is induced by the canonical map $d: \cc_{\Delta} \to f^*\dd_{\Delta}$, the second horizontal map comes from the natural equality $f^*\dd_{\Delta} \otimes_{\cc\!-\!\cc} \cc_{\Delta} \stackrel{\sim}{\leftarrow} f^*(\dd_{\Delta} \otimes_{\dd} \dd_{\Delta} \otimes_{\dd} \dd_{\Delta}) \otimes_{\cc\!-\!\cc} \cc_{\Delta} = \dd_{\Delta} \otimes_{\dd\!-\!\dd} (\dd_{\Delta} \otimes_{\cc} \cc_{\Delta} \otimes_{\cc} \dd_{\Delta}) = \dd_{\Delta} \otimes_{\dd\!-\!\dd} f_* \cc_{\Delta}$, the first and second vertical maps are induced applying $\alpha$ to the right $\cc_{\Delta}$ factor, the third vertical map induced by applying $f_*(\alpha)$, and the fourth induced by the commutative diagram, i.e., by (cohomologically) applying $c \circ f_{*}(\alpha) \circ c^{-1} = f_{\sharp}(\alpha)$.
\end{proof}

\begin{alem} \label{copairinglocalization}
    If $f: \cc \to \dd$ is a quotient functor (or more general homological epimorphism) with $\cc$ homologically smooth, then $\dd$ is smooth too, and $f_*$ sends the algebraic closed-string copairing for $\cc$ to the algebraic closed-string copairing for $\dd$.
\end{alem}
\begin{proof}[Proof Sketch]
    Up to isomorphism $f_*: [\cc,\cc] \to [\dd,\dd]$ sends  representables to representables, hence preserves perfection. It also sends the diagonal to the diagonal 
    establishing the first statement (compare \cite{efimov}*{Cor 3.5}). This further implies that the induced map $(f_*)_*: \r{HH}_*(perf(\cc\!-\!\cc)) \to \r{HH}_*(perf(\dd\!-\!\dd))$ sends $ch_{\cc_{\Delta}}$ to $ch_{\dd_{\Delta}}$, which is the key point needed for the second statement.
\end{proof}

The following result recalls that localizations or more general homological epimorphisms out of (weak) smooth Calabi-Yau categories inherit (weak) smooth Calabi-Yau structures:
\begin{alem}\label{localizeCY}
    If $f: \cc \to \dd$ is a quotient functor (or more general homological epimorphism) with $\cc$ and $\dd$ smooth, and $\sigma_{\cc}$ is a (weak) smooth Calabi-Yau structure, then $f_*(\sigma_{\cc})$ is a (weak) smooth Calabi-Yau structure on $\dd$, and there is a commutative diagram (cohomologically) 
    \[
        \xymatrix{ \r{HH}_{*-n}(\cc) \ar[r]^{f_*}  & \r{HH}_{*-n}(\dd)  \\
        \r{HH}^*(\cc) \ar[u]^{\cap {\sigma_{\cc}}}_{\cong} \ar[r]^{f_{\sharp}} & \r{HH}^*(\dd) \ar[u]^{\cap {f_*\sigma_{\cc}}}_{\cong}}
    \]
    where $f_{\sharp}$ is the algebra map of Lemma \ref{localizehh}.
\end{alem}
\begin{proof}[Proof Sketch]
 We already know from Lemma \ref{localizehh} that if $\cc$ is smooth so is $\dd$. Now  
    any functor $f$ induces a commutative diagram 
    \[
        \xymatrix{ \r{HH}_*(\cc) \ar[d]^{f_*} \ar[r]^{ev\hspace{10mm}} &  H^*(\hom_{\cc\!-\!\cc}(\cc_{\Delta}^!, \cc_{\Delta})) \ar[d]^{\Phi_f} \\
        \r{HH}_*(\dd) \ar[r]^{ev\hspace{10mm}} & H^*(\hom_{\dd\!-\!\dd}(\dd^!_{\Delta}, \dd_{\Delta})) },
\] 
(see e.g., \cite{bravdyckerhoff}*{Prop 4.3}, where the results are stated for dg categories with minor notational differences; the proof is the same in the $\ainf$ case or one can simply perform a dg replacement) where
$\Phi_f$ can be computed as the composite
\[
    \hom_{\cc\!-\!\cc}(\cc^!_{\Delta}, \cc_{\Delta}) \stackrel{f_*}{\to} \hom_{\dd\!-\!\dd}(f_*(\cc_{\Delta})^!, f_*(\cc_{\Delta})) \to \hom_{\dd\!-\!\dd}(\dd_{\Delta}^!, \dd_{\Delta}) ;
\]
for the middle term space we are using \cite{bravdyckerhoff}*{Lemma 4.2} to equate $f_*(\cc_{\Delta}^!)$ with  $f_*(\cc_{\Delta})^!$ and the second arrow is induced by the canonical map $c: f_*(\cc_{\Delta}) \to \dd_{\Delta}$, which is an isomorphism by hypothesis.
Since the first arrow (the map on morphism spaces induced by a functor) always sends isomorphisms to isomorphisms, we conclude $\Phi_f$ sends isomorphisms to isomorphisms. In light of the commutative diagram, we conclude that if $\sigma \in \r{HH}_{-n}(\cc)$ is a (weak) smooth Calabi-Yau structure, then so is $f_*(\sigma) \in \r{HH}_{-n}(\dd)$.

The second statement (commutative diagram) is an immediately corollary of the assertion from Lemma \ref{localizehh} that with respect to $f_{\sharp}$, $f_*$ is a map of $\r{HH}^*(\cc)$-modules, i.e., for any $\sigma \in \r{HH}_*(\cc)$ and $\alpha \in \r{HH}^*(\cc)$, $f_*(\alpha \cap \sigma) = f_{\sharp}(\alpha) \cap f_*(\sigma)$.
\end{proof}

\begin{alem}\label{opencopairingquotient}
    In the setting of Lemma \ref{localizeCY}, applying $f$ to morphism spaces sends the open-string algebraic copairing $c^{\sigma}_{K,L}$ on $(\cc, \sigma_\cc)$ to the open-string algebraic copairing $c^{f_*\sigma_\cc}_{fK,fL}$ on $(\dd, f_*\sigma)$.
\end{alem}
\begin{proof}
    All such $f$ considered are cohomologically unital by convention, i.e., $f$ sends $[id_L]$ to $[id_{fL}]$. Hence, the Lemma
    is an immediate consequence of verifying the commutative diagram, for any pair of objects (applying $f_*$ and using the isomorphism between $f_*\cc$ and $\dd$ and between $f_* \cc^!$ and $\dd^!$ for the two leftmost vertical arrows):
\[
    \xymatrix@R+1pc@C+2pc{ 
        H^*(\cc_{\Delta}(K,K)) \ar[r]^{ (ev(\sigma_{\cc}))_{K,K}^{-1}} \ar[d] & H^*(\cc^!(K,K)) \ar[r] \ar[d] & H^*(\hom_{\cc}(K,L)) \otimes H^*(\hom_{\cc}(L,K)) \ar[d]^{[f] \otimes [f]} \\
        H^*(\dd_{\Delta}(fK,fK)) \ar[r]^{ (ev(f_*\sigma_{\cc}))_{fK,fK}^{-1}} & H^*(\dd^!(fK,fK)) \ar[r] & H^*(\hom_{\dd}(fK,fL)) \otimes H^*(\hom_{\dd}(fL,fK)). 
    }
\]
\end{proof}

\bibliography{nonproper}
\bibliographystyle{alpha}

\end{document}